\documentclass[10pt,a4paper]{amsart}
\usepackage{graphics,multirow,array,amsmath,amssymb,epsf}

\newtheorem{theorem}{Theorem}
\newtheorem{lemma}[theorem]{Lemma}
\newtheorem{corollary}[theorem]{Corollary}

\begin{document}

\title{Knots with small lattice stick numbers}
\author[Y. Huh]{Youngsik Huh}
\address{Department of Mathematics, College of Natural Sciences,
Hanyang University, Seoul 133-791, Korea}
\email{yshuh@hanyang.ac.kr}
\author[S. Oh]{Seungsang Oh}
\address{Department of Mathematics, Korea University,
Anam-dong, Sungbuk-ku, Seoul 136-701, Korea}
\email{seungsang@korea.ac.kr}

\thanks{This work was supported by the Korea Science and
Engineering Foundation (KOSEF) grant funded by the Korea
government (MOST) (No.~R01-2007-000-20293-0).}

\begin{abstract}
The lattice stick number of a knot type is defined to be the minimal number of straight line segments
required to construct a polygon presentation of the knot type in the cubic lattice.
In this paper, we mathematically prove that the trefoil knot $3_1$
and the figure-8 knot $4_1$ are the only knot types of lattice stick number less than 15,
which verifies the result from previous numerical estimations on this quantity.
\end{abstract}

\maketitle

\section{Introduction} \label{sec:intro}
A circle embedded into the Euclidean 3-space $\mathbb{R}^3$ is called a {\em knot}.
Two knots $K$ and $K^{\prime}$ are said to be {\em ambient isotopic},
if there exists a continuous map $h:\mathbb{R}^3\times[0,1] \rightarrow \mathbb{R}^3$
such that the restriction of $h$ to each $t \in [0,1]$, $h_t:\mathbb{R}^3\times\{t\} \rightarrow \mathbb{R}^3$,
is a homeomorphism, $h_0$ is the identity map and $h_1(K_1)=K_2$,
to say roughly, $K_1$ can be deformed to $K_2$ without intersecting its strand.
The ambient isotopy class of a knot $K$ is called the {\em knot type} of $K$.
Especially if $K$ is ambient isotopic to another knot contained in a plane of $\mathbb{R}^3$,
then we say that $K$ is {\em trivial}.

A {\em polygonal knot\/} or {\em polygon} is a knot
which consists of line segments, called {\em sticks}.
A {\em lattice knot} is a polygon in the cubic lattice
$\mathbb{Z}^3=(\mathbb{R} \times \mathbb{Z} \times
\mathbb{Z}) \cup (\mathbb{Z} \times \mathbb{R} \times \mathbb{Z})
\cup (\mathbb{Z} \times \mathbb{Z} \times \mathbb{R})$.
A lattice knot of the knot type $3_1$ is depicted in Figure \ref{fig:1}-(a).
Figure \ref{fig:1}-(b) and (c) are showing lattice knots of the knot type $4_1$.

The polygon presentation of knots in $\mathbb{Z}^3$ has been considered to be a useful model
for simulating circular molecules, because it is simple and possesses a volume.
For founding studies on lattice knots the readers are referred to \cite{Ge, Ha1, Ha2, HW, Ke, Pi, SW}.
A quantity that we may naturally be interested on lattice knots is the minimum length necessary
to realize a knot type as a lattice knot which is called the {\em minimum step number} of the knot type.
This quantity is related to the minimum number of chemical components for molecular chains
such as DNA and proteins to be formed as a knot of a given knot type.
The minimum step number was numerically estimated for various knot types \cite{Hu, JP1, Sch}.
And such estimation was mathematically confirmed for some small knots.
Diao proved that the minimum step number of any non-trivial knot type is at least 24
and only $3_1$ can be realized with 24 steps \cite{D1}.
Also it was reported that the minimum step number of $4_1$ and $5_1$ are 30 and 34, respectively \cite{Sch}.

In this paper we deal with another quantity of lattice knots.
The {\em lattice stick number} $s_{L}(K)$ of a knot type $K$ is defined to be the minimal number of sticks
required to construct a lattice knot of the knot type.
We may say that this quantity corresponds to the total curvature of smooth knots
which was firstly studied by Milnor \cite{Mi}.
He showed that the total curvature of any smooth knot is at least $4\pi$, if its knot type is non-trivial.
For a polygon in $\mathbb{R}^3$ its total curvature can be defined to the summation of all angles
between every pair of adjacent sticks.
This quantity says how much the modeled polymer turn in space,
therefore is expected to have some connection with physical properties of molecular chains.
In \cite{Plu}, using numerical simulations,
the total curvature of polygons in $\mathbb{R}^3$ was scaled
as a function of the length of polygons for each knot type up to six crossings.
Also it was reported that the equilibrium length with respect to total curvature,
which appears to be correlated to physical properties of macromolecules,
can be considered to be one of characteristics of each knot type under the experiment.
In lattice knots the angle between any two adjacent sticks is $\frac{\pi}{2}$.
Hence for lattice knots the total curvature of any knot type $K$ can be clearly defined
to be $\frac{\pi}{2} s_{L}(K)$ with no necessity to consider the length.
Furthermore the restriction on the position of sticks establishes
some combinatorial arguments which may allow theoretical study on the quantity
other than computational simulations.

From the definition we easily know that the lattice stick number of trivial knot is 4.
Rensburg and Promislow proved that $s_L(K) \ge 12$ for any nontrivial knot $K$ \cite{JP2}.
Also they estimated the quantity for various knot types via the simulated annealing technique.
Another numerical estimation was performed in \cite{Sch}. And in \cite{HO},
it was mathematically proved that  $s_L(3_1) = 12$ and $s_L(K) \geq 14$
for any other non-trivial knot $K$ via a simple and elementary argument, called {\em properly leveledness}.

According to the results by estimation in \cite{JP2, Sch},
it can be conjectured that $s_L(K) \geq 15$ for any non-trivial knot $K$ except for $3_1$ and $4_1$.
In this paper, by extending the approach in \cite{HO}, we will verify the conjecture.

The following theorem is the main result of this paper.

\begin{theorem} \label{thm:1}
$s_L(K)=14$ if and only if $K$ is $4_1$.
\end{theorem}

\noindent By combining Theorem \ref{thm:1} and the result in \cite{HO} we have Corollary \ref{coro:1}.
\begin{corollary} \label{coro:1}
$s_L(K) \geq 15$ for every non-trivial knot $K$ except for $3_1$ and $4_1$.
\end{corollary}

The rest of this paper is devoted to the proof of Theorem \ref{thm:1}.

\begin{figure}[h]
\centerline{\epsfbox{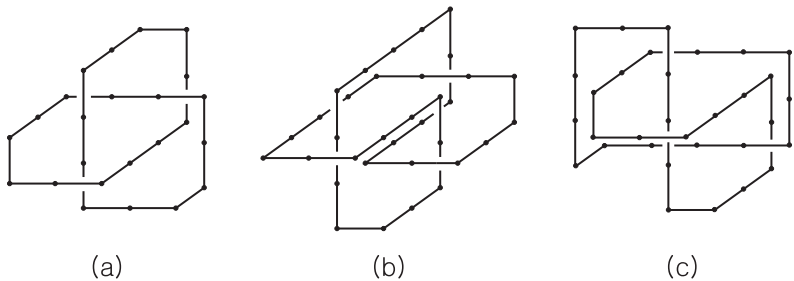}}
\caption{$3_1$ and $4_1$ in the cubic lattice}
\label{fig:1}
\end{figure}

\section{Proof of Main Theorem}

As a preparation for the proof of Theorem \ref{thm:1},
we briefly introduce some definitions and a lemma
which are found in the previous paper of the authors \cite{HO}.

Throughout this paper, a polygon will always mean a polygonal knot in $\mathbb{Z}^3$.
Two polygons are said to be {\em equivalent\/} if they are ambient isotopic in $\mathbb{R}^3$.
A polygon $P$ is called {\em reducible\/} if there is another equivalent polygon
which has fewer sticks.
Otherwise, it is {\em irreducible\/}.

Let $|P|$ denote the number of sticks of a polygon $P$.
A stick in $P$ which is parallel to the $x$-axis({\em resp.} $y$, $z$-axis)
is called an {\em $x$-stick\/}({\em resp.} {\em $y$, $z$-stick}) of $P$,
and $|P|_x$({\em resp.} $|P|_y$, $|P|_x$) denotes
the number of its $x$-sticks({\em resp.} {\em $y$, $z$-sticks }).
Each $y$-stick or $z$-stick lies on a plane whose $x$-coordinate is some integer $k$.
This plane is called an {\em $x$-level\/} $k$.
If $P$ has $n$ $x$-levels, then, without loss of generality,
we may say that these are $x$-levels $1,2,\cdots,n$ like heights.
In particular $x$-levels $1$ and $n$ are considered as boundary $x$-levels.
Note that an $x$-stick whose endpoints lie on $x$-levels $i$ and $j$ has length $|i-j|$.
A polygon $P$ is said to be {\em properly leveled\/}
if each $x$-level ({\em resp.} $y$, $z$-level) contains exactly two endpoints of $x$-sticks
({\em resp.} $y$, $z$-sticks).
Note that a properly leveled polygon $P$ has $|P|_x$ $x$-levels, $|P|_y$ $y$-levels and $|P|_z$ $z$-levels.
Three lemmas in Section 2 of \cite{HO} which are modified are essential in our study.

\begin{lemma} \cite{HO} \label{lem:1}
\begin{enumerate}
\item For a polygon $P$, there is a properly leveled polygon $P'$ equivalent to $P$ with $|P'| = |P|$.
\item Let $P$ be a properly leveled polygon.
If two $x$-sticks of $P$ have their endpoints on the same $x$-levels,
then the knot type of $P$ must be trivial (similarly for the $y$ or $z$-sticks).
\item Let $P$ be a properly leveled irreducible non-trivial polygon.
Then each boundary level contains only one stick or two sticks
which are connected.
Furthermore, if a stick of $P$ has one endpoint on a boundary level,
then this stick has length at least two.
\end{enumerate}
\end{lemma}

Now we prove Theorem \ref{thm:1}.
First of all, the existence of the lattice $4_1$ knots with $14$ sticks
as depicted in Figure \ref{fig:1}-(b) and (c) guarantees that $s_L(4_1) \leq 14$.
Let $P$ be an irreducible nontrivial lattice knot with $|P|\leq 14$.
By Lemma \ref{lem:1}-(1) we can assume that $P$ is properly leveled.
Then, by Lemma \ref{lem:1}-(2) and (3), we have $|P|_x,\;|P|_y,\;|P|_z \geq 4$.
Also it can be assumed that $|P|_x \geq |P|_y \geq |P|_z$.
Hence $(|P|_x, |P|_y, |P|_z)$ is equal to one of $(4,4,4)$, $(5,4,4)$, $(6,4,4)$ and $(5,5,4)$.
In \cite{HO} the case $|P|_y = |P|_z = 4$ was investigated, and the followings were proved;

\begin{enumerate}
\em
\item If $(|P|_x, |P|_y, |P|_z) = (4,4,4)$, then $P$ is $3_1$.
\item $(|P|_x, |P|_y, |P|_z) \neq (5,4,4)$.
\item If $(|P|_x, |P|_y, |P|_z) = (6,4,4)$, then $P$ is $4_1$.
\end{enumerate}

Thus the following lemma completes the proof of Theorem \ref{thm:1}.
Note that $P$ has $14$ sticks only in two cases $(6,4,4)$ and $(5,5,4)$.

\begin{lemma} \label{lem:main}
Let $P$ be a properly leveled irreducible non-trivial polygon.

\noindent If $(|P|_x, |P|_y, |P|_z)=(5,5,4)$, then $P$ is $4_1$.
\end{lemma}

The proof of this lemma will be given in the next section.

\section{Proof of Lemma \ref{lem:main}}

Since $P$ is properly leveled and $|P|_z =4$,
it has only four $z$-levels and therefore only four $z$-sticks,
namely $z_{14}$, $z_{13}$, $z_{24}$ and $z_{23}$,
where each $z_{ij}$ denotes the $z$-stick between two $z$-levels $i$ and $j$.
For brevity of notation, let $Z_i$ denote the $z$-level $i$ and
$P_i$ denote the subarc of $P$ which is the intersection of $P$ and $Z_i$ for each $i$.
And $P^{\circ}_i$ will denote the interior of $P_i$.
Note that $P$ is the union of eight arcs which are listed as below along an orientation of $P$;
\[P_1 \rightarrow z_{13} \rightarrow P_3 \rightarrow z_{23}
 \rightarrow P_2 \rightarrow z_{24} \rightarrow P_4 \rightarrow z_{14}. \]

\begin{lemma} \label{lemma:proj}
Let $p: \mathbb{R}^3 \rightarrow \mathbb{R}^2$ be the projection into the $xy$-plane.
Then the intersection between the two simple arcs $p(P^{\circ}_2)$ and $p(P^{\circ}_3)$
should be transverse.
\end{lemma}

\begin{proof}
For contradiction,
firstly we assume that the two arcs meet locally at a single point in a non-transverse manner.
Then the intersection point should be a projection of two vertices of $P$.
See Figure \ref{fig:2}-(a).
The four sticks near these vertices are an $x$-stick and a $y$-stick of $P_2$,
and an $x$-stick and a $y$-stick of $P_3$ whose endpoints on the other sides are $a, c, b,$ and $d$ respectively.
Because of the proper leveledness of $P$,
the two vertices $a$ and $b$ should be connected by a portion of $P$
which is contained in a single $y$-level.
In fact the connecting arc should contain more than one of the $z$-sticks
which are $z_{24}$, $z_{14}$, and $z_{13}$.
Similarly for the connection of $c$ and $d$, we need more than one $z$-stick,
which is a contradiction.

Now assume that their intersection contains a line segment.
Without loss of generality we may assume that the related two sticks are $y$-sticks.
Because of the proper leveledness of $P$, the two $y$-sticks are connected by a portion of $P$,
namely $L$, which is contained in a single $x$-level, that is,
$L$ includes only $y$-sticks and $z$-sticks.
If $L$ contains $z_{23}$, then it should be $z_{23}$ itself.
This implies that the number of sticks of $P$ can be reduced by $1$ as depicted in Figure \ref{fig:2}-(b),
contradicting the irreducibility of $P$.
Thus $L$ consists of three $z$-sticks $z_{24}$, $z_{14}$ and $z_{13}$,
and two more $y$-sticks between them alternately.
Now consider a connected portion of $P$ which is obtained by deleting $L$ and
two $y$-sticks attached to $L$.
The resulting portion consists of one $z$-stick, one $y$-stick and five $x$-sticks.
Thus some $x$-sticks of $P$ are adjacent, which is a contradiction.
\end{proof}

\begin{figure}[h]
\centerline{\epsfbox{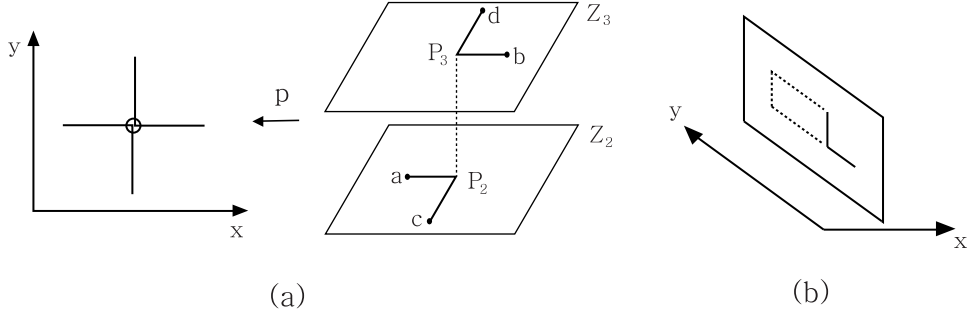}}
\caption{deleting non-transverse intersections}
\label{fig:2}
\end{figure}

\begin{lemma} \label{lemma:n2}
$p(P^{\circ}_2 \cup P^{\circ}_3)$ has at least two transverse double points.
\end{lemma}

\begin{proof}
First assume that it has no such point.
Pull down $P_3$ into the plane $Z_2$ along the $z$-axis.
To be a new polygon, remove $z_{23}$,
and shorten $z_{13}$ so that two ends of a new $z$-stick lie on $Z_1$ and $Z_2$.
The resulting polygon represents the same knot as $P$, but has fewer sticks.
This is a contradiction to the irreducibility of $P$.

Now assume that it has only one transversal double point.
$p(P_2 \cup P_3)$ separates the $xy$-plane
into two regions $D_1$ and $D_2$ one of which is unbounded.
Then $p(z_{13})$ and $p(z_{24})$ lie inside the same region, say $D_1$.
If another point $p(z_{14})$ lies inside $D_1$, we can isotope two subarcs $P_1$ and $P_4$ of $P$
on the planes $Z_1$ and $Z_4$ respectively
so that their projections do not meet each other and also $p(P_2 \cup P_3)$
except for their endpoints.
This implies that the knot $P$ has a diagram with only one crossing,
that is, $P$ is trivial.
If $p(z_{14})$ lies inside $D_2$, we can similarly isotope $P_1$ and $P_4$
so that their projections do not meet each other and each projection meets transversely
$p(P_2 \cup P_3)$ exactly once.
Thus the knot $P$ has a diagram with $3$ crossings, that is,
$P$ is trivial or $3_1$.
Both cases contradict the irreducibility of $P$.
\end{proof}

Now we observe $p(P_2 \cup P_3)$ on the square lattice $\mathbb{Z}^2$ of the $xy$-plane
which is $(\mathbb{R} \times \mathbb{Z}) \cup (\mathbb{Z} \times \mathbb{R})$,
more precisely on $([1,5] \times \{1,2,3,4,5\}) \cup (\{1,2,3,4,5\} \times [1,5])$.
Also we indicate an over/under-crossing at each crossing point
so that $p(P_3)$ goes over $p(P_2)$.
Suppose that there is a pair of subarcs of $p(P_2)$ and $p(P_3)$
so that they meet exactly twice,
one arc consists of only one stick and the other consists of three sticks,
and the region bounded by these two subarcs has no place to put a $z$-stick
in its interior as in Figure \ref{fig:3}.
Then we isotope $P$ by using a Reidermeister move II and
rearrange the related $x$ or $y$-levels so that $P$ is still properly leveled.
From now on we assume that there is no such pair of subarcs.

\begin{figure}[h]
\centerline{\epsfbox{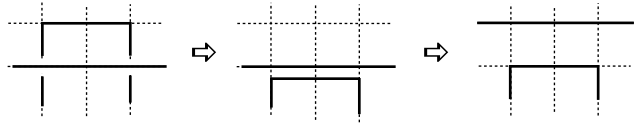}}
\caption{Reidermeister move II to simplify}
\label{fig:3}
\end{figure}

\begin{lemma} \label{lemma:no}
$p(P^{\circ}_2 \cup P^{\circ}_3)$ does not contain any of two types of pairs of subarcs
as in Figure \ref{fig:4}.
\end{lemma}

\begin{figure}[h]
\centerline{\epsfbox{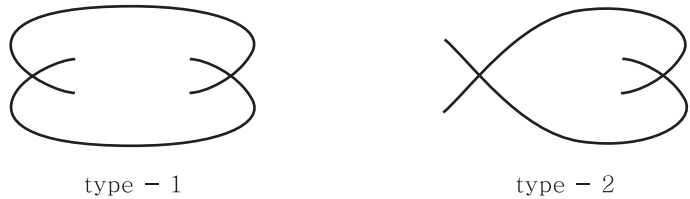}}
\caption{unwanted case}
\label{fig:4}
\end{figure}

\begin{proof}
Assume for contradiction that it contains a pair of subarcs of type-1 or 2.
To simplify the proof, we also assume that both subarcs of the pair
intersect each other on their end sticks.
Let $Q$ be the pair of subarcs of $P_2 \cup P_3$ whose projection is the previous pair
of subarcs of type-1 or 2.

First we will show that $Q$ consists of either six or eight sticks.
Each subarc is an array of $x$-sticks and $y$-sticks connected alternately.
Two sticks producing a double point are exactly one $x$-stick and one $y$-stick.
Then obviously $Q$ uses the same number of $x$-sticks and $y$-sticks,
so an even number of sticks.
Clearly a $Q$ of either type needs at least six sticks.
Also the number of sticks of $P_2 \cup P_3$ is at most eight
because $P$ has exactly four $z$-sticks,
and each of $P_1$ and $P_4$ on the boundary levels consists of either one or two sticks.

Now consider the other two subarcs of $P-Q$.
One subarc consists of $z_{23}$ and some $x$ or $y$-sticks, and
the other consists of $z_{13}$, $z_{14}$, $z_{24}$ and some $x$ or $y$-sticks.
In the latter case we need at least two $x$ or $y$-sticks
because any two $z$-sticks can not be adjoined.
Note that if $Q$ has eight sticks, then $P-Q$ consists of exactly one $x$-stick,
one $y$-stick and four $z$-sticks. Thus one component of $P-Q$ is just one stick $z_{23}$.

Without loss of generality we say $|P_3| \leq |P_2|$.
First assume that $Q$ is of type-1.
Then it consists of eight sticks.
Since one subarc of $P-Q$ is just $z_{23}$,
$p(P_2 \cup P_3)$ is one of the two cases as drawn in Figure \ref{fig:5}.
As mentioned in the statement before Lemma \ref{lemma:no},
we disregard the cases producing such a pair of subarcs as in the leftmost figure in Figure \ref{fig:3}.
In each figure $z_{23}$, $z_{13}$ and $z_{24}$ are placed at the points $a$, $b$ and $c$, respectively.
Furthermore $p(P_3)$ is the subarc from $a$ to $b$
and $p(P_2)$ is the arc from $a$ to $c$.
In both cases $P_1$ and $P_4$ are just one stick.
Thus $z_{14}$ must be placed at either $d$ or $e$ in the left case,
and either $f$ or $g$ in the right case.
But $z_{14}$ can not be placed at $d$, $e$, and $f$.
Also if $z_{14}$ is placed at $g$ in the right case,
then $P$ is reducible, a contradiction.

\begin{figure}[h]
\centerline{\epsfbox{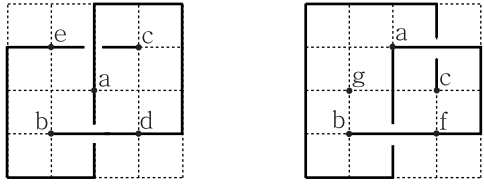}}
\caption{realizing type-1 case}
\label{fig:5}
\end{figure}

Now assume that $Q$ is of type-2.
If $Q$ consists of eight sticks, one subarc of $P-Q$ is just $z_{23}$.
This means that $p(P_2 \cup P_3)$ is obtained from $p(Q)$ by adding one point $p(z_{23})$.
But this is impossible because two end points of each subarc of $p(Q)$ lie on
different sides of the circle in $p(Q)$.
Therefore $Q$ consists of six sticks,
and so $p(P_2 \cup P_3)$ is one of the four cases as drawn in Figure \ref{fig:6}.
In each figure the arc from $a$ to $b$ is a subarc of $p(P_3)$
and the arc from $c$ to $d$ is a subarc of $p(P_2)$.
Since $Q$ consists of $3$ $x$-sticks and $3$ $y$-sticks,
$P-Q$ consists of $2$ $x$-sticks, $2$ $y$-sticks and $4$ $z$-sticks.
In each of the four figures, if one subarc of $P-Q$ connects $a$ and $c$,
it must consist of one $x$-stick, one $y$-stick and some $z$-sticks
since we also need at least one $x$-stick and one $y$-stick to connect $b$ and $d$.
Therefore this subarc must go through the point $e$, so $P$ is reducible.
Thus one subarc of $P-Q$ should run from $a$ to $d$, and also it must go through the point $f$.
There are two possibilities.
This subarc is either consisting of $3$ sticks
which are an $x$-stick, a $y$-stick and $z_{23}$ in this order
or consisting of $5$ sticks
which are $z_{13}$, an $x$-stick, $z_{14}$, a $y$-stick and $z_{24}$ in this order.
In both cases $P$ is reducible.
\end{proof}

\begin{figure}[h]
\centerline{\epsfbox{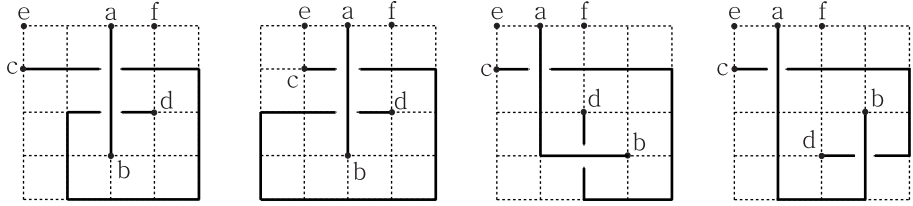}}
\caption{realizing type-2 case}
\label{fig:6}
\end{figure}

Now we prove Lemma \ref{lem:main}.
By Lemma \ref{lemma:no}, $p(P_2 \cup P_3)$ after deleting an arc portion of itself near $p(z_{23})$
is of the form with $n$ crossings as in Figure \ref{fig:7}-(a).
Let $A_1, \dots , A_{n-1}$ be the $n-1$ regions bounded by these two arcs.
Now draw the whole $p(P_2 \cup P_3)$ by adding the deleted portion.
Since $p(z_{23})$ should be located in the unbounded region outside
$A_1 \cup \cdots \cup A_{n-1}$,
the added arc divides the region into two regions $B_1$ and $B_2$,
and the two points $p(z_{13})$ and $p(z_{24})$ lie in $B_1 \cup B_2$.
Note that any $B_i$ which contains $p(z_{13})$ or $p(z_{24})$ is adjacent to every $A_j$.
Figure \ref{fig:7}-(b) illustrates a specific example of this addition where
both $p(z_{13})$ and $p(z_{24})$ lie in $B_1$ and $n$ is equal to 3.

\begin{figure}[h]
\centerline{\epsfbox{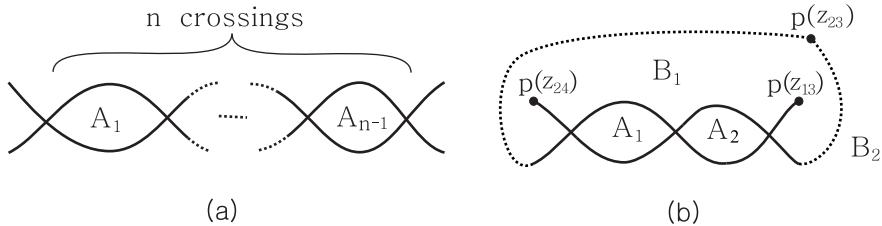}}
\caption{realizing $p(P_2 \cup P_3)$}
\label{fig:7}
\end{figure}

We consider the position of $p(z_{14})$.
First suppose that $p(z_{14})$ lies in either $B_1$ or $B_2$.
Then, as illustrated in Figure \ref{fig:8},
we can isotope each of the subarcs $P_1$, $P_3$ and $P_4$ along its respective $z$-level
so that $p(P^{\circ}_3)$ meets $p(P^{\circ}_2)$ in at most one point,
and each of $p(P^{\circ}_1)$ and $p(P^{\circ}_4)$ meets $p(P^{\circ}_2 \cup P^{\circ}_3)$
in at most one point.
The latter is possible because $B_1$ and $B_2$ are adjacent to each other.
This means that $P$ has a diagram with at most $3$ crossings, a contradiction.

\begin{figure}[h]
\centerline{\epsfbox{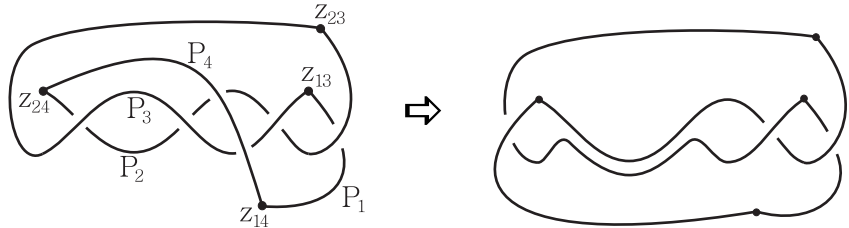}}
\caption{when $p(z_{14})$ lies in either $B_1$ or $B_2$}
\label{fig:8}
\end{figure}

Therefore $p(z_{14})$ lies in one of $A_i$'s, say $A_k$.
Then similarly we can isotope each of the subarcs $P_1$, $P_3$ and $P_4$
so that $p(P^{\circ}_3)$ meets $p(P^{\circ}_2)$ in at most two points
because $p(z_{14})$ is inside $A_k$,
and each of $p(P^{\circ}_1)$ and $p(P^{\circ}_4)$ meets $p(P^{\circ}_2 \cup P^{\circ}_3)$
in at most one point because $A_k$ is adjacent to any $B_i$ which contains $p(z_{13})$ or $p(z_{24})$.
Therefore $P$ has a diagram with at most $4$ crossings.
This implies $P$ should be $4_1$ because $3_1$ would be reducible.

This completes the proof.

\end{document}